\documentclass[preprint]{elsarticle}

\usepackage{amssymb}
\usepackage{amsmath}
\usepackage{amsthm}
\usepackage{enumerate}
\usepackage{graphicx}
\usepackage{amsfonts}

\journal{...}

\newtheorem{theorem}{Theorem}[section]

\newtheorem{cor}{Corollary}[section]

\newtheorem{remark}{Remark}[section]

\numberwithin{equation}{section}

\theoremstyle{definition}

\begin{document}

	\begin{frontmatter}

		\title{Elementary hypergeometric functions, Heun functions, and   moments of MKZ operators}
		
		\author[1]{Ana-Maria Acu\footnote{corresponding author}}
		\author[2]{Ioan Rasa}
		\address[1]{Lucian Blaga University of Sibiu, Department of Mathematics and Informatics, Str. Dr. I. Ratiu, No.5-7, RO-550012  Sibiu, Romania, e-mail: acuana77@yahoo.com}
			\address[2]{Technical University of Cluj-Napoca, Faculty of Automation and Computer Science, Department of Mathematics, Str. Memorandumului nr. 28 Cluj-Napoca, Romania,
	e-mail:  ioan.rasa@math.utcluj.ro }

		\begin{abstract} We consider some hypergeometric functions and prove that they are elementary functions. Consequently, the second order moments of
			Meyer-K\"onig and Zeller type operators are elementary functions. The
			higher order moments of these operators are expressed in terms of
			elementary functions and polylogarithms. Other applications are concerned with the expansion of certain Heun functions in series or finite sums of elementary hypergeometric functions.
			
		\end{abstract}

		
	\end{frontmatter}
\section{Introduction }\label{s1} This paper is devoted to some families of elementary hypergeometric functions, with applications to the moments of Meyer-K\"onig and Zeller type operators and to the expansion of certain Heun functions in series or finite sums of elementary hypergeometric functions.

 In \cite{2}, J.A.H. Alkemade proved that the second order moment of the Meyer-K\"onig and Zeller operators can be expressed as
\begin{equation}\label{e1.1}
M_n(e_2;x)=x^2+\dfrac{x(1-x)^2}{n+1}{}_{2}F_{1}(1,2;n+2;x),\, x\in[0,1).
\end{equation}
Here $e_r(t):=t^r$, $t\in[0,1]$, $r\geq 0$, and ${}_{2}F_{1}(a,b;c;x)$ denotes the hypergeometric function.

A closed form of $M_n(e_2;x)$, showing that it is an elementary function, was given in \cite{7}. It is said in \cite[p.2]{7}:
" The opinion that the second-order moment of the celebrated  Meyer-K\"onig  and Zeller operators is a non-elementary function was tacitly accepted since 1960." The same opinion was firmly expressed in \cite[Sect. 2]{D}.

This is surprising, because it is elementary to prove that ${}_{2}F_{1}(1,2;n+2;x)$ is elementary: such proofs will be given in the next sections. Alternatively, we can use
\begin{equation}\label{e1.2}
{}_{2}F_{1}(1,2;3;x)=-2x^{-2}(x+\log(1-x)),
\end{equation}
and \cite[2.8(24)]{6}:
\begin{equation}\label{e1.3}
{}_{2}F_{1}(a,b;c+m;x)=\dfrac{(c)_m(1-x)^{m+c-a-b}}{(c-a)_m(c-b)_m}\left((1-x)^{a+b-c}{}_{2}F_{1}(a,b;c;x)\right)^{(m)},
\end{equation}
where $(r)_m:=r(r+1)\dots(r+m-1),\,m\geq 1$, and $(r)_0:=1$.

They lead to
\begin{align}\label{e1.4}
{}_{2}F_{1}(1,2;n+2;x)&=\dfrac{(3)_{n-1}(1-x)^{n-1}}{(2)_{n-1}(1)_{n-1}}\left({}_{2}F_{1}(1,2;3;x)\right)^{(n-1)}\nonumber\\
&=-\dfrac{n+1}{(n-1)!}(1-x)^{n-1}\left(x^{-1}+x^{-2}\log(1-x)\right)^{(n-1)},
\end{align}
which is obviously an elementary function. Thus $M_n(e_2;x)$ given by (\ref{e1.1}) is elementary.

On the other hand, \cite[(32)]{3} provides  a closed form of the (elementary) hypergeometric function ${}_{2}F_{1}(1,m;m+2k+1;x)$, where $m\geq 1$ and $k\geq 0$ are integers.

In Section 2 we consider the more general function ${}_{2}F_{1}(m,n;p;x)$ with $m$, $p$ positive integers, $p\geq m+1$, and $n\in {\mathbb R}$. We present a closed form of this function, showing that it is elementary. A general formula for ${}_{2}F_{1}(a,b;a+b+l;x)$ can be found in \cite[2.3(2)]{6}, but for arbitrary parameters it is not always an elementary function.
Generalizing \cite[(32)]{3}, we give in Section 3 two closed forms for the elementary function
${}_{2}F_{1}(1,m;m+l+1;x)$ with $m\geq 1$, $l\geq 0$ integers. Three closed forms for ${}_{2}F_{1}(1,2;n+2;x)$ are presented in Section 4. Using a result of \cite{1} we give in Section 5 another proof that the second MKZ moment is elementary. The same result of \cite{1} is instrumental in Section 6, where we show that the higher order MKZ moments can be expressed as finite sums of functions, but containing polylogarithms. In the final section we consider two modifications of the MKZ operators and show that their second order moments are elementary functions; here we use Theorem \ref{t2.1} and the fact that ${}_{2}F_{1}(1,3;n+3;x)$ is elementary.

Let us mention that the moments of the MKZ operators are involved in
studying the asymptotic behavior of the iterates of these operators. For
a qualitative result, see \cite{D}. Generalized qualitative and quantitative
results can be found in \cite[Sect. 2]{B}, \cite{C} and \cite[Sect.1.3]{A}.

Section \ref{s8} is devoted to some applications involving Heun equations. The power-series solutions of such an equation are governed by three-term recurrence relations between the successive coefficients, instead of two-term recurrence relations which appear in the hypergeometric case. There are few Heun equations for which the exact solutions are known in terms of simpler functions. This is why expansions of the solutions of Heun equations using other functions instead of powers were considered by many authors; see \cite{C1}, \cite{R5}, \cite{R16}, \cite{R15},  and the references therein. 
In particular, expansions in terms of hypergeometric functions and conditions for deriving finite-sum solutions  can be found in \cite{C1}.

Using results from \cite{C1}, we describe a family of Heun equations for which the solutions are expanded in series of elementary hypergeometric functions. Moreover, we present conditions under which such a series terminates and so the solutions are elementary functions. Other results and applications will be given in a forthcoming paper.

\section{The function ${}_{2}F_{1}(m,n;p;x)$}\label{s2}

\begin{theorem}\label{t2.1}
	Let $m$, $p$ be positive integers, $p\geq m+1$, and $n\in {\mathbb R}$.  Then 
	\begin{align}\label{e2.1}
	{}_{2}F_{1}(m,n;p;x)&=\dfrac{(m)_{p-m}x^{1-p}}{(p-m-1)!}\displaystyle\sum_{k=0}^{m-1}(-1)^k{m-1\choose k}\nonumber\\
	&\cdot \displaystyle\sum_{j=0}^{p-m-1}{p-m-1\choose j}(-1)^jx^{p-m-1-j}\displaystyle\sum_{i=0}^j{j\choose i}(-1)^i\displaystyle\int_{1-x}^1s^{i+k-n}ds.
	\end{align}
	Consequently, ${}_{2}F_{1}(m,n;p;x)$ is an elementary function.
\end{theorem}
\begin{proof}
	According to the definition,
	$${}_{2}F_{1}(m,n;p;x)=\displaystyle\sum_{j=0}^{\infty}\dfrac{(m)_j(n)_j}{(p)_j j!}x^j.$$
	This implies
	\begin{align}\label{e2.2}
	\left(x^{p-1}{}_{2}F_{1}(m,n;p;x)\right)^{(p-m)}
	&=\displaystyle\sum_{j=0}^{\infty}\dfrac{(m)_j}{(p)_j}{n+j-1\choose j}\left(x^{j+p-1}\right)^{(p-m)}\nonumber\\
	&=(m)_{p-m}x^{m-1}\displaystyle\sum_{j=0}^{\infty}{n+j-1\choose j}x^j=(m)_{p-m}\dfrac{x^{m-1}}{(1-x)^n}.
	\end{align}
	For the function $x^{p-1}{}_{2}F_{1}(m,n;p;x)$, the Taylor polynomial of degree $p-2\geq p-1-m$ at 0 vanishes, and so
	\begin{align*}
	{}_{2}F_{1}(m,n;p;x)&=\dfrac{1}{x^{p-1}}\dfrac{1}{(p-m-1)!}\displaystyle\int_0^t(x-t)^{p-m-1}\left(t^{p-1}{}_{2}F_{1}(m,n;p;t)\right)^{(p-m)}dt\\
	&=\dfrac{1}{x^{p-1}}\dfrac{1}{(p-m-1)!}\displaystyle\int_0^t(x-t)^{p-m-1}(m)_{p-m}t^{m-1}(1-t)^{-n}d.
	\end{align*}
	Therefore,
	\begin{align*}
	{}_{2}F_{1}(m,n;p;x)=\dfrac{1}{x^{p-1}}\dfrac{(m)_{p-m}}{(p-m-1)!}\displaystyle\sum_{k=0}^{m-1}(-1)^k{m-1\choose k}\int_0^x(x-t)^{p-m-1}(1-t)^{k-n}dt.
	\end{align*}
	Now use the binomial formula for $(x-t)^{p-m-1}$ and set $s:=1-t$ in order to verify that the last integral equals
	$$ \displaystyle\sum_{j=0}^{p-m-1}{p-m-1\choose j}(-1)^j x^{p-m-1-j}\displaystyle\sum_{i=0}^j{j\choose i}(-1)^i\displaystyle\int_{1-x}^1s^{i+k-n} ds. $$
	This concludes the proof of (\ref{e2.1}).
\end{proof}
\begin{remark}\label{R2.1} The proof of Theorem \ref{t2.1} is truly elementary: it uses only the
definition of the hypergeometric function and Taylor's formula. Another
proof has been recently given in \cite{abel}.\end{remark}

For $m=1$, Theorem \ref{t2.1} yields 

\begin{cor}\label{c2.2}
	Let $p\geq 2$ be an integer, and $n\in {\mathbb R}$. Then
	\begin{equation}\label{e2.3}
	{}_{2}F_{1}(1,n;p;x)=\dfrac{p-1}{x^{p-1}}\displaystyle\sum_{i=0}^{p-2}{{p-2}\choose i}(x-1)^{p-2-i}\displaystyle\int_{1-x}^1 s^{i-n} ds.
	\end{equation}
	\end{cor}
\begin{proof}
	Setting $m=1$ in (\ref{e2.1}), we get
	\begin{align*}
	{}_{2}F_{1}(1,n;p;x)&=\dfrac{p-1}{x^{p-1}}\displaystyle\sum_{j=0}^{p-2}{p-2\choose j}(-1)^jx^{p-2-j}\displaystyle\sum_{i=0}^j{j\choose i}(-1)^i
\displaystyle\int_{1-x}^1	s^{i-n} ds\\
&=\displaystyle\dfrac{p-1}{x^{p-1}}\displaystyle\sum_{i=0}^{p-2}(-1)^i{p-2\choose i}\displaystyle\sum_{j=i}^{p-2}{p-2-i\choose j-i}(-1)^j x^{p-2-j}\displaystyle\int_{1-x}^1 s^{i-n}ds\\
&=\displaystyle\dfrac{p-1}{x^{p-1}}\displaystyle\sum_{i=0}^{p-2}{p-2\choose i}(x-1)^{p-2-i}\displaystyle\int_{1-x}^1 s^{i-n} ds.
\end{align*}
\end{proof}

\begin{remark}
	Let $\lambda,\mu$, and $\nu$ be nonnegative integers. Theorem \ref{t2.1} may be considered also in analogy to the following result (see \cite[p. 69]{6}, \cite[(23)]{A1}):
	\begin{align}
&	{}_{2}F_{1}(\nu+1,\nu+\mu+1;\nu+\mu+\lambda+2;x)\nonumber\\
	&=\dfrac{(-1)^{\mu+1}(\nu+\mu+\lambda+1)!}{\lambda!\nu!(\nu+\mu)!(\mu+\lambda)!}\dfrac{d^{\nu+\mu}}{dx^{\nu+\mu}}\left((1-x)^{\mu+\lambda}\dfrac{d^{\lambda}}{dx^{\lambda}}\left(\dfrac{\log(1-x)}{x}\right)\right).\label{e2.4}
	\end{align}
	Setting $m=\nu+1$, $n=\nu+\mu+1$, $p=\nu+\mu+\lambda+2$, (\ref{e2.4}) leads to 
		\begin{align}
	&	{}_{2}F_{1}(m,n;p;x)\nonumber\\
	&=(-1)^{n-m+1}\dfrac{p-1}{(p-2)!}{p-2\choose n-1}{p-2\choose m-1}\dfrac{d^{n-1}}{dx^{n-1}}\left((1-x)^{p-m-1}\dfrac{d^{p-n-1}}{dx^{p-n-1}}\left(\dfrac{\log(1-x)}{x}\right)\right),\label{e2.5}
	\end{align}
	for all integers $m,n,p$ with $p\geq n+1\geq m+1\geq 2$.
\end{remark}

\section{The function 	${}_{2}F_{1}(1,m;m+l+1;x)$}\label{s3}
\begin{theorem}	
	Let $m\geq 1$ and $l\geq 0$ be integers. Then 
	\begin{align}
	{}_{2}F_{1}(1,m;m+l+1;x)&=\dfrac{(m)_{l+1}}{x^{m+l}}\dfrac{(-1)^{l+1}}{l!}(1-x)^l\log(1-x)\nonumber\\
	&+\dfrac{m+l}{x^{m+l}}\displaystyle\sum_{i=0,i\ne m-1}^{m+l-1}{m+l-1\choose i}
	\dfrac{(-1)^{m+l-1-i}}{i-m+1}\left((1-x)^{m+l-1-i}-(1-x)^l\right),\label{e3.1}
	\end{align}
	and also
	\begin{align}
	{}_{2}F_{1}(1,m;m+l+1;x)&=\dfrac{(m)_{l+1}}{x^{m+l}}\dfrac{(-1)^{l+1}}{l!}(1-x)^l\log(1-x)\nonumber\\
	&+\dfrac{(m)_{l+1}}{x^{m+l}}\left\{\dfrac{(-1)^{l+1}}{l!}\displaystyle\sum_{j=1}^l x^j\sum_{i=0}^{j-1}\dfrac{(-1)^i}{j-i}{l\choose i}-\sum_{i=0}^{m-2}\dfrac{x^{l+i+1}}{(i+1)_{l+1}}\right\}.\label{e3.2}
	\end{align}
\end{theorem}
\begin{proof}
	(\ref{e3.1}) is a consequence of (\ref{e2.3}).
	
	To prove (\ref{e3.2}), let us replace in (\ref{e2.2}) $m$ by 1, $n$ by $m$, an $p$ by $m+l+1$; thus we have
	$$\left(x^{n+l} {}_{2}F_{1}(1,m;m+l+1;x)\right)^{(l+1)}=(m)_{l+1}\dfrac{x^{m-1}}{1-x}=(m)_{l+1}\left(\dfrac{1}{1-x}-1-x-\cdots-x^{m-2}\right).  $$
	It follows that 
	$$ x^{m+l}{}_{2}F_{1}(1,m;m+l+1;x)=(m)_{l+1}\left\{(-1)^{l+1}\dfrac{(1-x)^l}{l!}\log(1-x)-\displaystyle\sum_{i=0}^{m-2}\dfrac{x^{l+1+i}}{(i+1)_{l+1}}-Q_l(x)\right\}, $$
	where $Q_l$ is a polynomial of degree less or equal to $l$.
	
	The right hand side should be divisible by $x^{m+l}$, hence $Q_l(x)$ is the Taylor  polynomial of degree $l$, at 0, for the function $(-1)^{l+1}\dfrac{(1-x)^l}{l!}\log(1-x)$. It is not difficult to obtain
	$$ Q_l(x)=\dfrac{(-1)^l}{l!}\displaystyle\sum_{j=1}^{l}x^j\sum_{i=0}^{j-1}
	\dfrac{(-1)^i}{j-i}{l\choose i}. $$
	Putting all things together we get (\ref{e3.2}).
\end{proof}
\begin{remark}
	Equations (\ref{e3.1}) and (\ref{e3.2}) extend Eq. (32) in \cite{3}.
\end{remark}
\section{The function ${}_{2}F_{1}(1,2;n+2;x)$}\label{s4} For this function, which is involved in the representation (\ref{e1.1}), we have three explicit expressions:
\begin{theorem}
	Let $n\geq 1$ be an integer.Then
	\begin{align}
	{}_{2}F_{1}(1,2;n+2;x)&=(-1)^n\dfrac{n+1}{x^{n+1}}\left\{(1-x)^{n-1}\left[
		x+n\log(1-x)\right]-\displaystyle\sum_{j=1}^{n-1}(-1)^j\dfrac{n-j}{j}x^j(1-x)^{n-j-1}\right\},\label{e4.1}\\
		{}_{2}F_{1}(1,2;n+2;x)&=(-1)^n\dfrac{n+1}{x^{n+1}}\left\{(1-x)^{n-1}\left[x+n\log(1-x)\right]\right.\nonumber\\
	&\left.+\displaystyle\sum_{i=2}^n\dfrac{(-1)^i}{i-1}{n\choose i}\left[(1-x)^{n-i}-(1-x)^{n-1}\right]\right\},\label{e4.2}\\
	{}_{2}F_{1}(1,2;n+2;x)&=(-1)^n\dfrac{n(n+1)}{x^{n+1}}\left\{(1-x)^{n-1}\log(1-x)+\displaystyle\sum_{j=1}^{n-1}x^j\sum_{i=0}^{j-1}\dfrac{(-1)^i}{j-i}{n-1\choose i}\right\}-\dfrac{n+1}{x}.\label{e4.3}
	\end{align}
\end{theorem}
\begin{proof}
	(\ref{e4.1}) can be obtained if we continue the calculation in (\ref{e1.4}). (\ref{e4.2}) is a consequence of (\ref{e3.1}), and (\ref{e4.3}) is obtained from (\ref{e3.2}).
\end{proof}

\begin{remark}
	Other expressions for $	{}_{2}F_{1}(1,m;m+l+1;x)$ and $	{}_{2}F_{1}(1,2;n+2;x)$ can be obtained from (\ref{e2.5}).
\end{remark}

\section{The second MKZ moment is elementary: another proof}\label{s5}
From \cite[(6) and (7)]{1} we know that
$$ M_ne_r(x)=1+(1-x)^{n+1}\displaystyle\sum_{j=1}^r{r\choose j}\dfrac{(-n)^j}{(j-1)!}\int_0^{\infty}\dfrac{t^{j-1}e^{-nt}}{(1-xe^{-t})^{n+1}}dt. $$
On the other hand,
\begin{align*}
\displaystyle\int_0^{\infty} t^{j-1}e^{-nt}(1-xe^{-t})^{-n-1}dt &=\int_0^{\infty}t^{j-1}e^{-nt}\displaystyle\sum_{k=0}^{\infty}{-n-1\choose k}(-1)^kx^ke^{-kt}dt\\
&=\displaystyle\sum_{k=0}^{\infty}{n+k\choose k}x^k\int_0^{\infty} t^{j-1}e^{-(n+k)t}dt=\displaystyle\sum_{k=0}^{\infty}{n+k\choose k}x^k\dfrac{\Gamma(j)}{(n+k)^j},
\end{align*}
so that 
\begin{equation}\label{e5.1}
M_n e_r(x)=1+(1-x)^{n+1}\displaystyle\sum_{j=1}^r{r\choose j}(-n)^jf_{n,j}(x),
\end{equation}
where
\begin{equation}\label{e5.2}
f_{n,j}(x):=\displaystyle\sum_{k=0}^{\infty}{n+k\choose k}\dfrac{1}{(n+k)^j}x^k,\, n\geq 1,\, j\geq 0.
\end{equation}
Let us remark that
\begin{equation}
f_{n,j}(0)=\dfrac{1}{n^j}.
\end{equation}
Moreover,
$$ \left(x^nf_{n,j}(x)\right)^{\prime}=\displaystyle\sum_{k=0}^{\infty}{n+k\choose k}\dfrac{1}{(n+k)^{j-1}}x^{n-1+k}=x^{n-1}f_{n,j-1},  $$
and so
\begin{equation}\label{e5.4}
f_{n,j}(x)=\dfrac{1}{x^n}\int_0^x t^{n-1}f_{n,j-1}(t) dt,\,j\geq 1.
\end{equation}
We have $f_{n,0}(x)=\dfrac{1}{(1-x)^{n+1}}$, and (\ref{e5.4}) yields
\begin{equation}\label{e5.5}
f_{n,1}(x)=\dfrac{1}{n(1-x)^n}.
\end{equation}
From (\ref{e5.4}) and (\ref{e5.5}) it is easy to obtain
\begin{equation}\label{e5.6}
f_{n,2}(x)=\dfrac{(-1)^{n-1}}{nx^n}\left\{\displaystyle\sum_{i=1}^{n-1}{n-1\choose i}\dfrac{(-1)^i}{i}\dfrac{1-(1-x)^i}{(1-x)^i}-\log(1-x)\right\}.
\end{equation}
Now (\ref{e5.1}), (\ref{e5.5}) and (\ref{e5.6}) show that $M_ne_2(x)$ is an elementary function.

\section{Higher order moments of the MKZ operators}\label{s6}
We recall the definition (see \cite{9}) of the dilogarithm
\begin{equation}\label{e6.1}
Li_2(x):=-\int_0^x\dfrac{\log(1-t)}{t}dt=\displaystyle\sum_{k=1}^{\infty}\dfrac{x^k}{k^2},\,\,|x|\leq 1,
\end{equation}
and the polylogarithm of order $n$,
\begin{equation}\label{e6.2}
Li_n(x):=\displaystyle\int_0^x\dfrac{Li_{n-1}(t)}{t}dt=\displaystyle\sum_{k=1}^{\infty}\dfrac{x^k}{k^n},\,\, n\geq 3.
\end{equation}
Using (\ref{e5.4}) and (\ref{e5.6}) we get
\begin{align*}
f_{n,3}(x)&=\dfrac{(-1)^{n-1}}{nx^n}\displaystyle\int_0^x\left\{\displaystyle\sum_{i=1}^{n-1}{n-1\choose i} \dfrac{(-1)^i}{i}\dfrac{1-(1-t)^i}{t(1-t)^i}-\dfrac{\log(1-t)}{t}\right\}dt\\
&=\dfrac{(-1)^{n-1}}{nx^n}\int_0^x\left\{\displaystyle\sum_{i=1}^{n-1}{n-1\choose i}\dfrac{(-1)^i}{i}\left(\displaystyle\sum_{l=0}^{i-2}(1-t)^{l-i}+\dfrac{1}{1-t}\right)-\dfrac{\log(1-t)}{t}\right\}dt,
\end{align*}
so that finally,
\begin{equation}
\label{e6.3} 
f_{n,3}(x)=\dfrac{(-1)^{n-1}}{nx^n}\left\{\displaystyle\sum_{i=1}^{n-1}{n-1\choose i}\dfrac{(-1)^i}{i}\left[\displaystyle\sum_{l=0}^{i-2}\dfrac{1-(1-x)^{i-l-1}}{(i\!-\!l\!-\!1)(1\!-\!x)^{i-l-1}}\!-\!\log(1\!-\!x)\right]+Li_2(x)\right\}.
\end{equation}
\begin{theorem}\label{t6.1}
	Let $n\geq 2$. For each $j\geq 2$, the functions $x^nf_{n,j}(x)$ are linear combinations, with constant coefficients, of $n$ functions as follows:
	\begin{itemize}
		\item [1)] For $2\leq j\leq n$:
		$$ \dfrac{1-(1-x)^i}{(1-x)^i},\,\, i=1,\dots,n-j+1;\,\, \log(1-x);\,Li_2(x),\dots,Li_{j-1}(x).  $$
		\item [2)] For $j=n+1$:
		$$ \log(1-x);\,\, Li_2(x),\dots, Li_n(x). $$
		\item [3)] For $j\geq n+2$:
		$$Li_{j-n}(x),\dots, Li_{j-1}(x).  $$
	\end{itemize}
\end{theorem}
\begin{proof}
	We use induction on $j$. The first step, involving $j=2$, $j=3$, is accomplished by using (\ref{e5.6}) and (\ref{e6.3}). Then the induction is continued with the recurrence relation (\ref{e5.4}) and (\ref{e6.1}), (\ref{e6.2}).
\end{proof}
\begin{remark}
	Using (\ref{e5.1}) and Theorem \ref{t6.1} we get information about the structure of the moments $M_ne_r$. In particular, we see once more that $M_ne_2$ is elementary. The higher order moments contain polylogarithms.
\end{remark}

\section{Other operators}\label{s7} Let $m_{n,k}(x):={n+k\choose k}(1-x)^{n+1}x^k$, $x\in[0,1]$, and for $f\in L_1[0,1]$,
$$ L_n(f;x):=\displaystyle\sum_{k=0}^{\infty}\dfrac{(n+k+1)(n+k+2)}{n+1}\left(\displaystyle\int_0^1m_{n,k}(t)f(t)dt\right)m_{n,k}(x). $$

The operators $L_n$ were investigated in \cite{5}; they coincide with the operators $M_{n,2}$ from \cite{8}; details can be found in \cite{4}.

It was proved in \cite{5} (see also \cite[p. 123]{4}) that
\begin{equation}
\label{e7.1} L_n(e_2;x)=x^2+\dfrac{2x(1-x)^2}{n+2}{}_2F_1(1,3;n+3;x).
\end{equation}

Closed forms of the elementary functions ${}_2F_1(1,3;n+3;x)$ can be obtained from (\ref{e3.1}) and (\ref{e3.2}) taking $m=3$, $l=n-1$.

In particular, we see that $L_n(e_2;x)$ is an elementary function. It would be interesting to study from this point of view the higher order moments of $L_n$ and, more generally, of other modifications of the MKZ operators presented in \cite{4}.

\begin{remark} In \cite[Chapter 7]{Alkemade} the generalized Meyer-K\"onig and
	Zeller operators are defined as follows:	
	\begin{align*}
	\left(M_{n,r}^{\alpha,\beta}f\right)(x)&:=(1-x)^{n+r}\displaystyle\sum_{k=0}^{\infty}{n+r+k-1\choose k}x^k f\left(\dfrac{k+\beta}{n+k+\alpha}\right),\\
	&	x\in[0,1),\,\,\alpha,\beta\in{\mathbb R},\,\, \alpha\geq\beta\geq 0,\,\, r\in {\mathbb Z}, n,n+r\in {\mathbb N},\end{align*}
	and it is proved that
	\begin{align}\label{e7.2}
	\left(M_{n,r}^{\alpha,\beta}e_1\right)(x)&=\dfrac{\beta}{n+\alpha}{}_2F_1(1,\alpha-r;n+1+\alpha;x)+\dfrac{n+r-\beta}{n+1+\alpha}x{}_2F_1(1,\alpha-r+1,n+2+\alpha;x).
	\end{align}
	If $\alpha\geq 0$ is an integer, we can apply Theorem \ref{t2.1} to conclude that $M_{n,r}^{\alpha,\beta}e_1$ is an elementary function.
	
	Moreover, suppose that $\alpha\geq 0$ is an integer and $r=\alpha+1$. Then (\ref{e7.2}) yields
	\begin{equation}\label{e7.3}
	\left(M_{n,\alpha+1}^{\alpha,\beta}e_1\right)(x)=\dfrac{\beta}{n+\alpha}+\left(1-\dfrac{\beta}{n+\alpha}\right)x.
	\end{equation}
	Let  $g_j(x):=(1-x)^j,\,\, j=0,1,\dots.$ Using a trick invented by Professor Ulrich Abel (see \cite{*}), we can write
	\begin{align*}
	\left(M_{n,\alpha+1}^{\alpha,\beta}g_j\right)(x)&=\dfrac{(n+\alpha-\beta)^j}{(n+\alpha)!}(1-x)^{n+\alpha+1}\displaystyle\sum_{k=0}^{\infty}\dfrac{(n+k+\alpha)!}{k!}\dfrac{x^k}{(n+k+\alpha)^j}\\
	&=\dfrac{(n+\alpha-\beta)^j}{(n+\alpha)!}(1-x)^{n+\alpha+1}\left(\displaystyle\sum_{k=0}^{\infty}\dfrac{x^{n+k+\alpha}}{(n+k+\alpha)^j}\right)^{(n+\alpha)}\\
	&=\dfrac{(n+\alpha-\beta)^j}{(n+\alpha)!}(1-x)^{n+\alpha+1}\left(Li_j(x)\right)^{(n+\alpha)}.
	\end{align*}
	Combined with $e_m=\displaystyle\sum_{j=0}^m(-1)^j{m\choose j}g_j$, this leads to 
	$$ \left(M_{n,\alpha+1}^{\alpha,\beta}e_m\right)(x)=\dfrac{(1-x)^{n+\alpha+1}}{(n+\alpha)!}\displaystyle\sum_{j=0}^m (-1)^j{m\choose j}(n+\alpha-\beta)^j\left(Li_j(x)\right)^{(n+\alpha)}. $$
	Now it is easy to recover (\ref{e7.3}) and to see that $M_{n,\alpha+1}^{\alpha,\beta}e_2$ is an elementary function.

\end{remark}

\section{ Elementary Heun functions}\label{s8}
The general Heun equation in canonical form is  (see \cite{B1}, \cite{D1},   \cite{E1}):
\begin{equation}\label{e8.1}
\dfrac{d^2u}{dx^2}+\left(\dfrac{\gamma}{x}+\dfrac{\delta}{x-1}+\dfrac{\varepsilon}{x-a}\right)\dfrac{du}{dx}+\dfrac{\alpha\beta x-q}{x(x-1)(x-a)}u=0,
\end{equation}
where the parameters satisfy $\alpha+\beta+1=\gamma+\delta+\varepsilon$.

Under the hypotheses that $\alpha\beta\ne 0$ and $\gamma+\varepsilon$ is not zero or a negative integer, an expansion of the solution $u(x)$ as a series of hypergeometric functions is investigated in \cite{C1}.

Let $m$ and $p$ be positive integers, $p\geq m+1$, and $n\in {\mathbb R}$, $n\ne 0$. In (\ref{e8.1})
choose 
$$ \alpha=m,\,\,\beta=n,\,\, \gamma=p+1-m-n,\,\, \delta=m+n-p+1,  $$
$$ \varepsilon=m+n-1,\,\,a=\dfrac{1}{2},\,\, q=\dfrac{1}{2}\left(mn-(m+n-p)(m+n-1)\right).  $$

\begin{theorem}
	\label{t8.1} With the above parameters, the solution $u(x)$ of (\ref{e8.1}) is given by
	\begin{equation}\label{e8.2}
	u(x)=\displaystyle\sum_{k=0}^{\infty}\dfrac{\left(\dfrac{m+n-1}{2}\right)_k\left(\dfrac{p-m}{2}\right)_k\left(\dfrac{p-n}{2}\right)_k}{k!\left(\dfrac{p}{2}\right)_k\left(\dfrac{p+1}{2}\right)_k}{}_2F_1(m,n;p+2k;x),
	\end{equation}
	and ${}_2F_1(m,n;p+2k;x)$ are elementary functions. Moreover, if there exists a positive integer $r$ such that $m+n=3-2r$ or $n-p=2r-2$, then the series (\ref{e8.2}) reduces to a finite sum and so $u(x)$ is an elementary function.
\end{theorem}
\begin{proof}
	It is easy to check that $\alpha\beta\ne 0$, $\gamma+\varepsilon=p\geq 2$, $a=\dfrac{1}{2}$, $\gamma+\delta=2$, $q=a\alpha\beta+a(1-\delta)\varepsilon$, and so the conditions \cite[(43)]{C1} are satisfied. According to \cite[(45)]{C1}, the solution of the Heun equation (\ref{e8.1}) is given by
	$$ u(x)=\displaystyle\sum_{k=0}^{\infty}\dfrac{\left(\frac{\varepsilon}{2}\right)_k \left(\frac{\gamma+\varepsilon-\alpha}{2}\right)_k  \left(\frac{\gamma+\varepsilon-\beta}{2}\right)_k}{k!  \left(\frac{\gamma+\varepsilon}{2}\right)_k  \left(\frac{1+\gamma+\varepsilon}{2}\right)_k}\,\,{}_2F_1(\alpha,\beta;\gamma+\varepsilon+2k;x), $$
	which leads immediately to (\ref{e8.2}).
	
	Using Theorem \ref{t2.1} we see that ${}_2F_1(m,n;p+2k;x)$ is an elementary function for each $k\geq 0$.
	
	Let $m+n=3-2r$, or $n-p=2r-2$, with $r$ an integer, $r\geq 1$. Then
	$$ \left(\dfrac{m+n-1}{2}\right)_r=0,\textrm{ or } \left(\dfrac{p-n}{2}\right)_r=0,  $$
	and so the series in (\ref{e8.2}) reduces to the sum for $k\in\left\{0,1,\dots,r-1\right\}$; this means that in this case $u(x)$ is an elementary function for which explicit expressions can be obtained using results from Sections \ref{s2}-\ref{s4}.
\end{proof}
\begin{remark}
	Usually, a Heun function is normalized by $u(0)=1$. For the equation (\ref{e8.1}), \cite[(46)]{C1} yields
	$$ u(0)={}_3F_2\left(\dfrac{\gamma+\varepsilon-\alpha}{2}, \dfrac{\gamma+\varepsilon-\beta}{2}, \dfrac{\varepsilon}{2};\dfrac{\gamma+\varepsilon}{2}, \dfrac{1+\gamma+\varepsilon}{2};1\right),  $$
	where ${}_3F_2$ is Clausen's generalized hypergeometric function. With the parameters used in Theorem \ref{t8.1} we get
	$$ u(0)={}_3F_2\left(\dfrac{p-m}{2},\dfrac{p-n}{2},\dfrac{m+n-1}{2};\dfrac{p}{2},\dfrac{1+p}{2};1\right).  $$
	
	Thus $u(x)/u(0)$ satisfies the normalization condition.
\end{remark}

$  $

\noindent{\bf Acknowledgement.} 
The work of the first author was financed  by Lucian Blaga University of Sibiu $\&$ Hasso Plattner Foundation research grants LBUS-IRG-2019-05.

\end{document}